\documentclass{head/cernrep}

\usepackage{bbm}
\usepackage{xcolor}
\usepackage{mathrsfs}
\usepackage{amsthm}
\usepackage{natbib}
\usepackage{mathtools}
\usepackage{enumerate}
\usepackage{hyperref}
\usepackage{comment}
\usepackage{ulem}
\usepackage{algorithm}
\usepackage{algorithmic}

\newcommand{\dif}{\mathrm{d}}
\newcommand{\bP}{\mathbb{P}}
\newcommand{\bE}{\mathbb{E}}

\newcommand{\bC}{\mathbb{C}}
\newcommand{\bN}{\mathbb{N}}

\newcommand{\ind}{\mathbf{1}}
\newcommand{\supp}{\mathrm{supp}}

\newcommand{\cF}{\mathcal{F}}
\newcommand{\cG}{\mathcal{G}}

\theoremstyle{remark}
\newtheorem{remark}{Remark}[section]
\theoremstyle{plain}

\theoremstyle{plain}
\newtheorem{lemma}{Lemma}[section]
\theoremstyle{plain}
\newtheorem{theorem}{Theorem}[section]
\theoremstyle{plain}
\newtheorem{corollary}{Corollary}[section]
\theoremstyle{definition}

\begin{document}

\title{Fast computation and theoretical guarantees for the NPMLE in exponential family mixtures}
\author{Yan Zhang}
\institute{University of Toronto}

\begin{abstract}
This work makes two advances in the study of the (approximate) nonparametric maximum likelihood estimator (NPMLE) for exponential family mixture models. First, we develop a data-compression strategy that reduces the cost of repeated likelihood evaluations in NPMLE computation to polylogarithmic order in the sample size. Second, we show that, for a broad class of approximate NPMLEs, the resulting marginal density estimator attains an almost parametric rate of convergence.
\end{abstract}

\maketitle

\section{Introduction}

The nonparametric maximum likelihood estimator (NPMLE) has recently emerged as a standard 
tool in modern data analysis, playing a central role in methods for principal components analysis \citep{zhong2022empirical}, hierarchical linear modeling \citep{soloff2025multivariate}, and empirical partially Bayes multiple testing \citep{ignatiadis2025empirical}; see also \cite{koenker_gu_2026}. These advances rest on two complementary developments: a long line of work on the statistical properties of the NPMLE in mixture models, from \citet{lindsay1983geometry1,lindsay1983geometry2} to \citet{zhang2009generalized,saha2020nonparametric}, and substantial progress in computational methods and software \citep{koenker2014convex,zhang2024efficient,yan2024learning}. In this paper, we push this line of work further by studying the NPMLE for general exponential family mixtures. We propose a generic computational approach that substantially improves efficiency, and we develop a unified theoretical framework to better understand the strong statistical performance of the NPMLE.

In many applications, the data are naturally modeled as arising from mixtures of some parametric family \citep{lindsay1995mixture}. Let \(\{p_\theta : \theta \in \Theta\}\) be a family of densities on \(\mathbb{R}\), defined with respect to a common reference measure \(\mu\), where \(\Theta \subseteq \mathbb{R}\) is the parameter space. For a mixing distribution \(g\) on \(\Theta\), the corresponding mixture density is
\[
f_g(x) \coloneqq \int_{\Theta} p_\theta(x)\,\mathrm{d}g(\theta).
\]
Suppose we observe an i.i.d.\ sample \(X_1,\ldots,X_n \sim f_{g_0}\) for some unknown mixing distribution \(g_0\), and the goal is to recover certain features of \(g_0\). The NPMLE treats \(g_0\) as an unknown probability distribution supported on a known subset \(\Theta_0 \subseteq \Theta\), and estimates it by any maximizer
\begin{equation}\label{eq:NPMLE}
\hat g \in \arg\sup_{g\in\mathcal G}\,\ell_n(f_g),
\end{equation}
where
\begin{equation}\label{eq:mixing}
\mathcal G \coloneqq \{g : \mathrm{supp}(g)\subseteq \Theta_0\},
\end{equation}
and \(\ell_n(f) \coloneqq \sum_{i=1}^n \log f(X_i)\) denotes the log-likelihood. 

Because \(\Theta_0\) often contains infinitely many points, the exact NPMLE \(\hat g\) is generally computationally intractable. A variety of methods have been proposed to approximate it, including the EM algorithm \citep{jiang2009general}, interior-point methods \citep{koenker2014convex}, and the Frank--Wolfe algorithm \citep{feng2018approximate}; see also \citet{yan2024learning} and \citet{zhang2024efficient} for more recent computational developments in the multivariate setting. A common feature of all these approaches is that the log-likelihood \(\ell_n(f)\), or its dual formulation, must be evaluated repeatedly. Consequently, for important models such as Gaussian mixtures \citep{zhang2009generalized} and chi-square mixtures \citep{ignatiadis2025empirical}, the computational cost of each iteration is at least linear in \(n\), making these methods burdensome for large-scale problems\footnote{For example, Section 1.3 of \citet{ignatiadis2025empirical} analyzes a methylation dataset with \(n=439{,}918\).}. An important exception arises in many discrete models, such as those in which \(p_\theta\) is Poisson or geometric. Indeed, if \(\mu\) is the counting measure on the natural numbers \(\mathbb{N}\) (including zero), then
\begin{equation}\label{eq:trick}
\sum_{i=1}^n \log f(X_i)=\sum_{k=0}^{|X|_n} N_k \log f(k),
\end{equation}
where \(|X|_n \coloneqq \max_{i\in[n]} |X_i|\), \(N_k \coloneqq \#\{i\in[n]: X_i = k\}\), and \([n] \coloneqq \{1,\dots,n\}\). Consequently, when \(f_{g_0}\) is light-tailed, \(|X|_n\) may grow only logarithmically with \(n\), substantially reducing the cost of evaluating \(\ell_n(f)\). The first goal of this work is to extend this computational advantage beyond the discrete setting to more general models, in which the observations \(X_1,\dots,X_n\) are typically distinct.

Once the NPMLE \(\hat g\), or an approximation thereof, has been computed, a central theoretical question is how accurately the plug-in estimator \(f_{\hat g}\) approximates the true density \(f_{g_0}\). The existing literature is largely model-specific, with most results focusing on Poisson mixtures \citep{lambert1984asymptotic,shen2022empirical,jana2025optimal} and Gaussian location mixtures \citep{ghosal2001entropies,jiang2009general,saha2020nonparametric,soloff2025multivariate}. More recently, theoretical guarantees have also been established for scaled chi-square mixtures \citep{ignatiadis2025empirical}. 
A common feature of these results is that, when \(f_{g_0}\) has sufficiently light tails, \(f_{\hat g}\) converges to \(f_{g_0}\) at a nearly parametric rate, up to logarithmic factors, thereby illustrating a broad adaptivity property of the NPMLE. Beyond these model-by-model analyses, a unified theory is currently available only for a class of discrete models known as mixtures of power series distributions \citep{balabdaoui2025parametric}. Since these models are supported on \(\mathbb N\), they also automatically enjoy the computational advantages described above. Our second goal in this work is to develop a unified theoretical framework for the nearly parametric convergence of the (approximate) NPMLE in marginal density estimation over exponential family mixtures.

\section{Main results}

Following \citet{lindsay1983geometry1,lindsay1983geometry2}, we work with the following exponential family mixture model.

\begin{enumerate}
\renewcommand{\theenumi}{(EF)}
\renewcommand{\labelenumi}{\theenumi}
\item \label{(EF)}
Let \(p_0\) be a base density on \(\mathbb{R}\) with respect to a dominating measure \(\mu\). Define its moment generating function \(Z\) and cumulant generating function \(\kappa\) by
\[
Z(\theta)\coloneqq \int \exp(\theta x)\, p_0(x)\,\mathrm{d}\mu(x),
\qquad
\kappa(\theta)\coloneqq \log Z(\theta),
\]
for \(\theta\) in the canonical parameter space
\[
\Theta \coloneqq \{\theta \in \mathbb{R} : Z(\theta) < \infty\}.
\]
The corresponding exponential family is given by
\[
p_\theta(x)\coloneqq \exp\bigl(\theta x-\kappa(\theta)\bigr)\,p_0(x),
\qquad \theta\in\Theta.
\]
We assume that \(X_1,\dots,X_n\) are i.i.d. observations from \(f_{g_0}\), and that
\[
\supp(g_0)\subseteq \Theta_0 \coloneqq \{\theta\in\mathbb{R}: |\theta|\le M\}\subseteq \Theta
\]
for some known constant \(M\in(0,\infty]\).
\end{enumerate}

Notable examples of the exponential family in \ref{(EF)} include the following. 

\begin{enumerate}
\renewcommand{\theenumi}{(SC)}
\renewcommand{\labelenumi}{\theenumi}
\item \label{(SC)} \textit{Scaled chi-square mixture:} \(M<\infty\), and
\[
p_0(x)=\Bigl(\frac{\nu}{2\sigma^2}\Bigr)^{\nu/2}\frac{x^{\nu/2-1}}{\Gamma(\nu/2)}
\exp\Bigl(-\frac{\nu x}{2\sigma^2}\Bigr),
\]
for some fixed integer \(\nu\ge 2\) and constant \(\sigma^2>0\).

\renewcommand{\theenumi}{(GL)}
\renewcommand{\labelenumi}{\theenumi}
\item \label{(GL)} \textit{Gaussian location mixture:} \(M=\infty\), and
\[
p_0(x)=\frac{1}{\sqrt{2\pi}}\exp\Bigl(-\frac{x^2}{2}\Bigr).
\]
\end{enumerate}

To our knowledge, the study of the NPMLE for model~\ref{(SC)} was initiated only recently by \citet{ignatiadis2025empirical}, motivated by applications to multiple testing. By contrast, for model~\ref{(GL)}, convergence rate analyses have a longer history: density estimation results go back at least to \citet{ghosal2001entropies} and were further developed by \citet{zhang2009generalized}, while convergence rate results for empirical Bayes estimation were established by \citet{jiang2009general}. We view these two examples as representative of two broader scenarios:

\begin{enumerate}
\renewcommand{\theenumi}{(S1)}
\renewcommand{\labelenumi}{\theenumi}
\item \label{(S1)}
Assume \ref{(EF)} holds. In addition, suppose that \(M<\infty\) and that there exists \(\tau>0\) such that
\[
[-M-\tau,\,M+\tau]\subseteq \Theta.
\]

\renewcommand{\theenumi}{(S2)}
\renewcommand{\labelenumi}{\theenumi}
\item \label{(S2)}
Assume \ref{(EF)} holds. In addition, suppose that \(M=\infty\), that \(p_0\) has mean zero, and that there exist constants \(\alpha_1\ge \alpha_2>0\) such that
\begin{equation}\label{eq:CGF}
c_1|\theta|^{1+1/\alpha_1}\le \kappa(\theta)\le c_2|\theta|^{1+1/\alpha_2},
\qquad
|\theta|\ge c_0,
\end{equation}
for some constants \(c_0,c_1,c_2>0\).
\end{enumerate}

These two scenarios place assumptions on one of the two main ingredients of the mixture model: the mixing class \eqref{eq:mixing}, determined by \(M\), and the component family, determined by \(p_0\). Scenario~\ref{(S1)} assumes that the mixing distribution is supported on a compact interval, while allowing the baseline density \(p_0\) to be fairly general. In particular, it permits \(p_0\) to have tails of the same order as the exponential distribution, which is essentially the heaviest tail compatible with the exponential family framework. 

By contrast, scenario~\ref{(S2)} allows the mixing distribution to have
support on all of \(\mathbb{R}\), while imposing conditions on the tail
behavior of \(p_0\). In particular, it requires the tails of \(p_0\) to be
neither too heavy nor too light. The lower bound on \(\kappa(\theta)\)
rules out compactly supported \(p_0\), whereas the upper bound excludes
densities too close to exponential. A similar assumption appears in
Theorem~5 of \citet{polyanskiy2020self}.

\begin{remark}[Mixtures of power series distributions]
Our work is closely related to \citet{balabdaoui2025parametric}, who study
marginal density estimation via the NPMLE in an exponential-family setting
with \(\mu\) equal to the counting measure on \(\mathbb{N}\). Although our
framework covers a broader class of models, it is best viewed as
complementary to, rather than an extension of, their work.
Their analysis relies jointly on four assumptions, (A1)--(A4).
Assumptions~(A1) and~(A2) require \(g_0\) to have compact support, which is
closely related in spirit to scenario~\ref{(S1)}, except that no compact
interval containing the support of \(g_0\) is assumed to be known a priori.
Assumptions~(A3) and~(A4) impose that the tail of \(p_0\) be neither too
heavy nor too light. Compared with scenario~\ref{(S2)}, their setting is
restricted to a one-sided, relatively ``heavy-tailed'' regime, covering, for instance, the
Poisson and geometric distributions.
\end{remark}

Prior work on the NPMLE \eqref{eq:NPMLE} \citep{zhang2009generalized,ignatiadis2022confidence,jana2025optimal,balabdaoui2025parametric} suggests that a light-tailed $f_{g_0}$ is essential for $f_{\hat{g}}$ to exhibit near-parametric behavior. We therefore assume:
\begin{enumerate}
\renewcommand{\theenumi}{(LT)}
\renewcommand{\labelenumi}{\theenumi}
\item \label{(LT)} There exists $\beta_0>0$ such that
\[
\bP(|X|_n \ge t)\le n\exp(-b_1 t^{\beta_0}), \qquad t\ge b_0,
\]
for some $b_0,b_1>0$.
\end{enumerate}
In scenario~\ref{(S1)}, this condition is automatically satisfied with $\beta_0=1$. In scenario~\ref{(S2)}, by contrast, it requires additional assumptions on $g_0$. See Lemma~\ref{lem:tail} for a discussion.

\subsection{Approximate NPMLE via data compression}\label{sec:data-compression}

Under \ref{(EF)}, an important observation is that the NPMLE $\hat g\in\cG$ can equivalently be expressed as a maximizer of
\[
\int \log l_g(x)\,\dif P_n(x)=\frac{1}{n}\sum_{i=1}^n \log l_g(X_i),
\]
where $P_n$ denotes the empirical measure based on the observations, and $l_g(x)$ is the likelihood ratio function given by
\[
l_g(x)\coloneqq \int l_\theta(x)\,\dif g(\theta), \qquad
l_\theta(x)\coloneqq \frac{p_\theta(x)}{p_0(x)}=\exp\bigl(\theta x-\kappa(\theta)\bigr).
\]
In the discrete setting, the key computational idea behind \eqref{eq:trick} is to compress $P_n$ into a discrete measure whose number of support points is much smaller than $n$. Guided by this perspective, we propose the following procedure for handling general observations. 

Given a sequence of (possibly random) positive integers $\{J_n\}_{n\ge1}$, let \(P_{J_n}\) denote an order-\(J_n\) Gaussian quadrature rule for \(P_n\). By construction, \(P_{J_n}\) and \(P_n\) share the same first \(2J_n-1\) moments. The existence of \(P_{J_n}\) is standard; see Theorem 2A of \citet{lindsay1989moment}. Based on \(P_{J_n}\), we define the approximate NPMLE by
\begin{equation}\label{eq:hatgJn}
\hat{g}_{J_n}\in\arg\sup_{g\in\mathcal G}\int \log l_g(x)\,\mathrm{d}P_{J_n}(x).
\end{equation}
The construction is summarized in Algorithm~\ref{alg:gaussian-quadrature}. This procedure naturally includes the discrete setting discussed above. Indeed, it suffices to take \(J_n = |X|_n + 1\), in which case \(P_{J_n} = P_n\), and therefore \(\hat g_{J_n} = \hat g\). By contrast, in many important cases, such as \ref{(SC)} and \ref{(GL)}, one typically has \(P_{J_n}\neq P_n\) unless \(J_n=n\), where no computational savings are obtained. This motivates the question of how large \(J_n\) should be in order for \(\hat{g}_{J_n}\) to be sufficiently close to \(\hat g\) under an appropriate criterion.

In classical discussions of NPMLE computation, one seeks an estimator \(\tilde g\in\mathcal G\) satisfying the likelihood criterion
\begin{equation}\label{eq:LRT}
\ell_n(f_{\hat g})-\ell_n(f_{\tilde g})\le \Delta_n,
\end{equation}
where \(\{\Delta_n\}_{n\ge 1}\) is a sequence of tolerance levels. The presence of a positive gap \(\Delta_n\) in \eqref{eq:LRT} is computationally unavoidable because \(\mathcal G\) is infinite-dimensional. Given a candidate \(\tilde g\), one can often compute an upper bound on this likelihood gap without explicitly evaluating \(\hat g\); see Section 6.4 of \citet{lindsay1995mixture}. The central question is therefore how small \(\Delta_n\) must be for the approximate NPMLE \(\tilde g\) to inherit the key statistical properties of \(\hat g\). The answer is task-dependent: different inferential goals require different levels of likelihood accuracy. We give a brief discussion of known tolerance requirements in Section~\ref{sec:tolerance}.

Our main result characterizes the relationship between $J_n$ and $\Delta_n$ when $\tilde{g} = \hat{g}_{J_n}$.

\begin{theorem}\label{thm:approximateNPMLE1} Assume \ref{(EF)}. Let \(\{M_n\}_{n\ge1}\) be a sequence of positive numbers such that
\begin{equation}\label{eq:bound}
\mathrm{supp}(\hat g)\cup \mathrm{supp}(\hat g_{J_n})\subseteq [-M_n,M_n].
\end{equation}
Assume that \(|X|_n M_n \ge 1\). Then the likelihood criterion \eqref{eq:LRT} holds for $\tilde{g}=\hat{g}_{J_n}$\footnote{In practice, $\hat g_{J_n}$ is rarely computed exactly. Instead, an approximation error is introduced to account for its deviation from the true maximizer. Theorem~\ref{thm:approximateNPMLE2} formalizes this, yielding a stronger version of Theorem~\ref{thm:approximateNPMLE1}. } provided that
\[
J_n
=
\bigg\lceil
2|X|_nM_n
\log\Big(
\frac{Cn|X|_nM_n(|X|_nM_n+\sup_{|\theta|\le M_n}|\kappa(\theta)|)}{\Delta_n}
\Big)
\bigg\rceil,
\]
where \(C>0\) is a universal constant, and \(\lceil x\rceil\) denotes the smallest integer larger than or equal to \(x\in\mathbb R\).
\end{theorem}

Theorem~\ref{thm:approximateNPMLE1} applies under the basic exponential family setup, without any further assumptions. The following corollary shows that, under both scenarios considered here, polylogarithmic growth of \(J_n\) is sufficient to ensure a strong likelihood approximation by \(\hat g_{J_n}\).

\begin{corollary}\label{cor:order}
Fix some \(\gamma>0\). To guarantee \eqref{eq:LRT} for $\tilde{g}=\hat{g}_{J_n}$ with \(\Delta_n=n^{-\gamma}\), one may choose \(J_n\) as in Theorem~\ref{thm:approximateNPMLE1}. In that case,
\begin{itemize}
\item Under \ref{(S1)},
\[
J_n = O_{\mathbb P}\bigl((\log n)^2\bigr).
\]

\item Under \ref{(S2)} and \ref{(LT)},
\[
J_n = O_{\mathbb P}\bigl((\log n)^{1+(1+\alpha_1)/\beta_0}\bigr).
\]
\end{itemize}
\end{corollary}

 This corollary, however, does not by itself provide concrete guidance for the practical choice of $J_n$. In practice, we suggest choosing $J_n$ as the largest value for which $P_{J_n}$ can still be computed in a numerically stable manner. In our implementation, we employ the Golub--Welsch algorithm \citep{golub1969calculation}, which computes $P_{J_n}$ accurately and stably for $J_n$ up to about 20--30. This range is typically adequate even for sample sizes as large as $100{,}000$.

Although constructing \(P_{J_n}\) still incurs at least linear cost in \(n\),
this one-time preprocessing step is substantially cheaper than computing the
NPMLE on the full sample. In our simulations, it is more than two orders of
magnitude faster than the implementations of \citet{koenker2014convex} and
\citet{zhang2024efficient}; see Section~\ref{sec:simulations}.

Algorithm~\ref{alg:gaussian-quadrature} is presented in one dimension for
theoretical clarity, but the same idea extends to multivariate settings. As an
illustration, Section~\ref{sec:heter} treats the heteroscedastic Gaussian
sequence model, where observing both \(X_i\) and its variability measure makes
the computation essentially two-dimensional.

\subsection{Marginal density estimation with approximate NPMLE}\label{sec:marginal-density}

The purpose of this section is to show that, in the general exponential family settings \ref{(S1)} and \ref{(S2)}, estimators that attain sufficiently high sample likelihood in the sense of \eqref{eq:LRT} also enjoy strong guarantees for marginal density estimation. Earlier results of this kind were available only for particular models, such as Gaussian location mixtures \citep{zhang2009generalized}, and were formulated in terms of the squared Hellinger distance
\[
H^2(f,f_0)\coloneqq\int \big(\sqrt{f(x)}-\sqrt{f_0(x)}\big)^2\dif\mu(x),
\]
where $f$ and $f_0$ are densities with respect to $\mu$.

Traditional analyses of the NPMLE typically proceed by controlling the entropy of the class of marginal densities
\[
\{f_g : g \in \mathcal{G}\},
\]
with respect to an appropriately restricted supremum norm. In general exponential family settings, however, this approach becomes difficult to implement, since the base density \(p_0\) may be highly irregular. A different argument is therefore required.

Under scenario~\ref{(S1)}, the parameter space \(\Theta_0\) is bounded. Using the notation introduced there, we define the surrogate density
\[
\bar{f}_g(x)\coloneqq\int \bar{p}_\theta(x)\,\mathrm{d}g(\theta),
\qquad
\bar{p}_\theta(x)\coloneqq l_\theta(x)\exp\bigl(-(M+\tau)|x|\bigr).
\]
By controlling the size of the class
\[
\bar{\mathcal{F}}\coloneqq\{\bar{f}_g:g\in\mathcal{G}\},
\]
as established in Lemma~\ref{lem:entropy}, we obtain the following convergence result.

\begin{theorem}\label{thm:density-rate1}
Suppose that scenario~\ref{(S1)} holds, and let \(\tilde g\) be an approximate NPMLE satisfying \eqref{eq:LRT} with \(\Delta_n = A(\log n)^2\) for some constant \(A>0\). Then there exists a constant \(C_{p_0,M,A}>0\), depending only on \(p_0\), \(M\), and \(A\), such that
\[
\mathbb P\Big(
H^2(f_{\tilde g},f_{g_0})
\ge
t\,C_{p_0,M,A}\frac{(\log n)^2}{n}
\Big)
\le
\exp\bigl(-t(\log n)^2\bigr),
\]
for all \(t\ge 1\) and \(n\ge 1\).
\end{theorem}

\begin{remark}
The above result implies Theorem 9 of \citet{ignatiadis2025empirical}, which plays a central role in their theoretical analysis. In comparison with Theorem 2.2 of \citet{balabdaoui2025parametric}, we obtain a faster convergence rate together with a stronger tail bound. This improvement stems from the additional assumption that the support of $g_0$ is contained in a known compact interval $\Theta_0$ that is strictly smaller than the full parameter space $\Theta$.
\end{remark}

In scenario~\ref{(S2)}, we have \(\Theta_0=\Theta=\mathbb{R}\). Instead of introducing a surrogate density, we work directly with the class of likelihood ratio functions
\[
\mathcal{L} \coloneqq \{l_g : g\in\cG\}. 
\]
A corresponding entropy bound for this class, established in Lemma~\ref{lem:entropy2}, yields the following result.

\begin{theorem}\label{thm:density-rate2}
Assume that scenario~\ref{(S2)} and condition~\ref{(LT)} hold, and let \(\tilde g\) be an approximate NPMLE satisfying \eqref{eq:LRT} with \(\Delta_n=A(\log n)^{\gamma_0}\) for some constant \(A>0\). Then there exists a constant
\(C_{p_0,g_0,A}>0\), depending only on \(p_0\), \(g_0\), and \(A\), such that
\[
\mathbb P\Big(
H^2(f_{\tilde g},f_{g_0})
\ge
t\,C_{p_0,g_0,A}\frac{(\log n)^{\gamma_0}}{n}
\Big)
\le
2\exp\bigl(-t^{\gamma_1}\log n\bigr),
\]
for all \(t\ge1\) and \(n\ge C_{p_0,g_0,A}\), where
\[
\gamma_0
\coloneqq
2\Bigl(\frac{1+\alpha_1}{\beta_0}\vee 1\Bigr),
\qquad
\gamma_1
\coloneqq
\frac{\beta_0}{2(1+\alpha_1)}\wedge 1.
\]
\end{theorem}

\begin{remark}
The above result is somewhat weaker than model-specific results, such as Theorem~1 of \citet{zhang2009generalized} for Gaussian location mixtures. Its main strength, however, is its generality: it shows that the nearly parametric convergence rate of the (approximate) NPMLE extends well beyond individual models. In this way, it highlights the statistical relevance of the approximate NPMLE constructed in Section~\ref{sec:data-compression}.
\end{remark}

\section{Conclusions and discussion}

Focusing on a broad class of exponential family mixture models, this paper makes two main contributions. First, we propose a data-compression strategy that changes the sample-size dependence of likelihood evaluations in NPMLE computation from linear to polylogarithmic order. Beyond its computational benefit, this perspective also has a clear theoretical implication: the fitting procedure depends only on the information contained in the first few moments. The same construction extends naturally to multivariate settings; see Section~\ref{sec:heter} for an illustration in the heteroscedastic Gaussian sequence model. More broadly, the idea applies whenever one seeks to evaluate expectations of smooth functions with respect to an empirical distribution.

Second, we establish an almost parametric rate for the NPMLE in marginal density estimation. In contrast to the classical theory of \citet{ghosal2001entropies,zhang2009generalized}, our analysis relies on entropy control for auxiliary classes beyond the marginal density class and uses complex-analytic arguments, leading to a more streamlined proof. The resulting theory not only provides statistical support for the approximate NPMLE developed in the first part of the paper, but also suggests a general framework for analyzing more complicated, including multivariate, models on a case-by-case basis.

\section*{Acknowledgement}

The author thanks Xin Bing, Jiaying Gu, Ricardo Baptista and Wenlong Mou for helpful discussions and Stanislav Volgushev for substantial feedback during the preparation of this manuscript.

\bibliographystyle{apalike}
\bibliography{ref}

\section{Supplement} \label{sec:supplement}

\subsection{Remarks on likelihood tolerance levels}\label{sec:tolerance}

The statistical requirement on the likelihood gap \(\Delta_n\) in \eqref{eq:LRT} depends on the task. If \(\Delta_n=o(1)\), then \(\tilde g\) is asymptotically indistinguishable from \(\hat g\) from the perspective of likelihood ratio testing, and is therefore suitable for certain testing procedures, such as homogeneity tests \citep{azais2009likelihood,gu2018testing}. Moreover, for the Gaussian location model, \citet{zhang2009generalized} and \citet{jiang2009general} showed that the plug-in estimators based on \(\tilde g\) achieve nearly parametric performance in both density estimation and empirical Bayes estimation whenever \(\Delta_n\) is of logarithmic order, that is, \(\Delta_n = O\bigl((\log n)^\gamma\bigr)\) for some \(\gamma>0\). An analogous result is also available for the Poisson model \citep{jana2025optimal}. For these two models, our recent work \citet{zhang2026adaptivity} further shows that the plug-in estimators based on \(\tilde g\) attain exact parametric performance in these estimation tasks when \(\Delta_n = O(1)\), provided that the true mixing distribution \(g_0\) is finitely discrete and \(\Theta_0\) is bounded.

\subsection{Heteroscedastic Gaussian sequence model}\label{sec:heter}

This section extends the data-compression strategy from Section~\ref{sec:data-compression} to the heteroscedastic Gaussian sequence model \citep{jiang2020general}. 

\begin{enumerate}
\renewcommand{\theenumi}{(HG)}
\renewcommand{\labelenumi}{\theenumi}
\item \label{(HG)}
Consider
\[
X_i\mid(\theta_i,\sigma_i^2)\overset{\mathrm{ind.}}{\sim}\mathcal{N}(\theta_i,\sigma_i^2),
\qquad
\theta_i\overset{\mathrm{i.i.d.}}{\sim} g_0,
\qquad
i=1,\dots,n,
\]
where \(\{(X_i,\sigma_i^2)\}_{i\in[n]}\) are observed, and \(g_0\) is interpreted as a common prior distribution belonging to the class \eqref{eq:mixing} with 
$$
\Theta_0 \coloneqq \{\theta\in\mathbb{R}: |\theta|\le M\}. 
$$
\end{enumerate}

Under \ref{(HG)}, the marginal density of \(X_i\) given \(\sigma_i^2\) is
\[
f_{g_0}(x|\tau_i)\coloneqq \sqrt{\frac{\tau_i}{2\pi}}\int \exp\Big(-\frac{\tau_i(x-\theta)^2}{2}\Big)\,\dif g_0(\theta),
\]
where \(\tau_i\coloneqq 1/\sigma_i^2\). 
The NPMLE in this setting is defined by
\[
\hat{g}\coloneqq \arg\sup_{g\in\mathcal{G}} \int \log f_g(x|\tau)\,\dif P_n(x,\tau),
\]
where \(P_n\) is the empirical measure based on \(\{(X_i,\tau_i)\}_{i\in[n]}\).

As before, given a sequence of integers \(\{J_n\}_{n\ge 1}\), let
\[
P_{J_n}\coloneqq \text{a quadrature rule matching all moments of }P_n\text{ up to total degree }J_n.
\]
We construct this quadrature rule via Carath\'eodory--Tchakaloff subsampling \citep{piazzon2016caratheodory}. This procedure produces a quadrature rule with at most \(\binom{J_n+2}{2}\) support points and satisfying
\[
\mathrm{supp}(P_{J_n})\subseteq \mathrm{supp}(P_n).
\]
We then define
\[
\hat{g}_{J_n}\coloneqq \arg\sup_{g\in\mathcal{G}} \int \log f_g(x|\tau)\,\dif P_{J_n}(x,\tau).
\]
The following theorem extends Theorem~\ref{thm:approximateNPMLE1} and Corollary~\ref{cor:order} to the heteroscedastic Gaussian sequence model.

\begin{theorem}\label{thm:approximateNPMLE-hetero}
Assume \ref{(HG)}. Suppose moreover that, for some \(T_0>0\),
\[
\tau_i\in[1/T_0,T_0], \qquad i\in[n],
\]
and that
\[
|X|_n = O_{\bP}\big((\log n)^{\alpha_0}\big)
\]
for some \(\alpha_0>0\). To ensure that
\[
\sum_{i=1}^n\Bigl(\log f_{\hat g}(X_i| \tau_i)-\log f_{\hat g_{J_n}}(X_i| \tau_i)\Bigr)\le n^{-\gamma}
\]
for some \(\gamma>0\), it suffices to take
\[
J_n=\Bigl\lceil C_{T_0}|X|_n(|X|_n\wedge M)\log\bigl(C_{T_0}n^{1+\gamma}|X|_n^6\bigr)\Bigr\rceil,
\]
where \(C_{T_0}>0\) is a constant depending only on \(T_0\). In particular,
\begin{itemize}
\item if \(M<\infty\), then
\[
J_n = O_{\mathbb P}\bigl((\log n)^{1+\alpha_0}\bigr);
\]
\item if \(M=\infty\), then
\[
J_n = O_{\mathbb P}\bigl((\log n)^{1+2\alpha_0}\bigr).
\]
\end{itemize}
\end{theorem}

As shown in \citet{jiang2020general}, a likelihood guarantee of this form for $\hat{g}_{J_n}$ is already sufficient to obtain empirical Bayes risk guarantees.

\subsection{Simulations}\label{sec:simulations}

\begin{algorithm}[t]
\caption{Construction of the approximate NPMLE}
\label{alg:gaussian-quadrature}
\begin{algorithmic}[1]
\REQUIRE A probability measure \(P_n\) and a positive integer \(J_n\)
\STATE Compute the first \(2J_n-1\) moments of \(P_n\)
\STATE Construct a discrete probability measure \(P_{J_n}\), supported on \(J_n\) points, that matches these moments; for example, via the Golub--Welsch algorithm \citep{golub1969calculation}
\STATE Compute the maximizer in \eqref{eq:hatgJn}
\ENSURE An approximate NPMLE \(\hat g_{J_n}\)
\end{algorithmic}
\end{algorithm}

In this section, we consider the Gaussian location mixture model \ref{(GL)} and generate data $X_1,\dots,X_n$ from $f_{g_0}$, where $g_0$ is the uniform distribution on the interval $[-2,2]$. To approximate $\hat{g}$ and $\hat{g}_{J_n}$, we adopt a discretization approach, under which all computations are restricted to mixing distributions supported on a finite set of fixed grid points. In practice, we take 300 points uniformly spaced over the interval
\[
[\min_{i\in[n]} X_i,\ \max_{i\in[n]} X_i].
\]

Since \citet{koenker2014convex}, the interior-point method has been the standard implementation under this framework, and it is also adopted by the \texttt{R} package \texttt{REBayes} for empirical Bayes applications \citep{koenker2017rebayes}. More recently, \citet{zhang2024efficient} proposed an augmented Lagrangian method, which implicitly exploits the sparsity structure of the NPMLE \citep{polyanskiy2020self} and thereby substantially accelerates computation. We therefore compare the following three methods:
\begin{itemize}
    \item[(IPM)] Approximate the NPMLE \eqref{eq:NPMLE} using the interior-point method.
    \item[(ALM)] Approximate the NPMLE \eqref{eq:NPMLE} using the augmented Lagrangian method.
    \item[(Ours)] Implement Algorithm~\ref{alg:gaussian-quadrature}, while approximating \eqref{eq:hatgJn} using the interior-point method.
\end{itemize}

Algorithm~\ref{alg:gaussian-quadrature} is not a direct competitor to (IPM) or (ALM), but rather a complementary procedure containing an additional preprocessing step. In principle, our approach could also be combined with the augmented Lagrangian method to further reduce computation time. In our framework, however, the main computational bottleneck is the preprocessing step, so such a combination yields only modest additional gains. That said, this hybrid approach may still be useful when one wishes to compute the NPMLE under multiple models on the same dataset, as is often the case when comparing several competing specifications.

We consider the case $n=100{,}000$, set $J_n=25$ for our method, and repeat the experiment 10 times. Let $\tilde{g}$ denote a resulting estimator. Figure~\ref{fig:comparison} compares the three methods in terms of computation time (in seconds), log-likelihood $\ell_n(f_{\tilde{g}})$, and the sum of squared errors
\[
\sum_{i=1}^n \left(\theta_i-\frac{\int \theta\, p(X_i | \theta)\,\dif \tilde{g}(\theta)}{f_{\tilde{g}}(X_i)}\right)^2,
\]
where $\theta_i$ is the latent variable associated with $X_i$ under the data-generating model
\[
X_i \mid \theta_i \overset{\mathrm{ind.}}{\sim} \mathcal{N}(\theta_i,1),
\qquad
\theta_i \overset{\mathrm{i.i.d.}}{\sim} g_0,
\qquad
i=1,\dots,n.
\]
The figure shows that our approach is substantially faster than the other two methods, while achieving statistically indistinguishable performance.

\begin{figure}[htbp]
    \centering
    \includegraphics[width=0.32\textwidth]{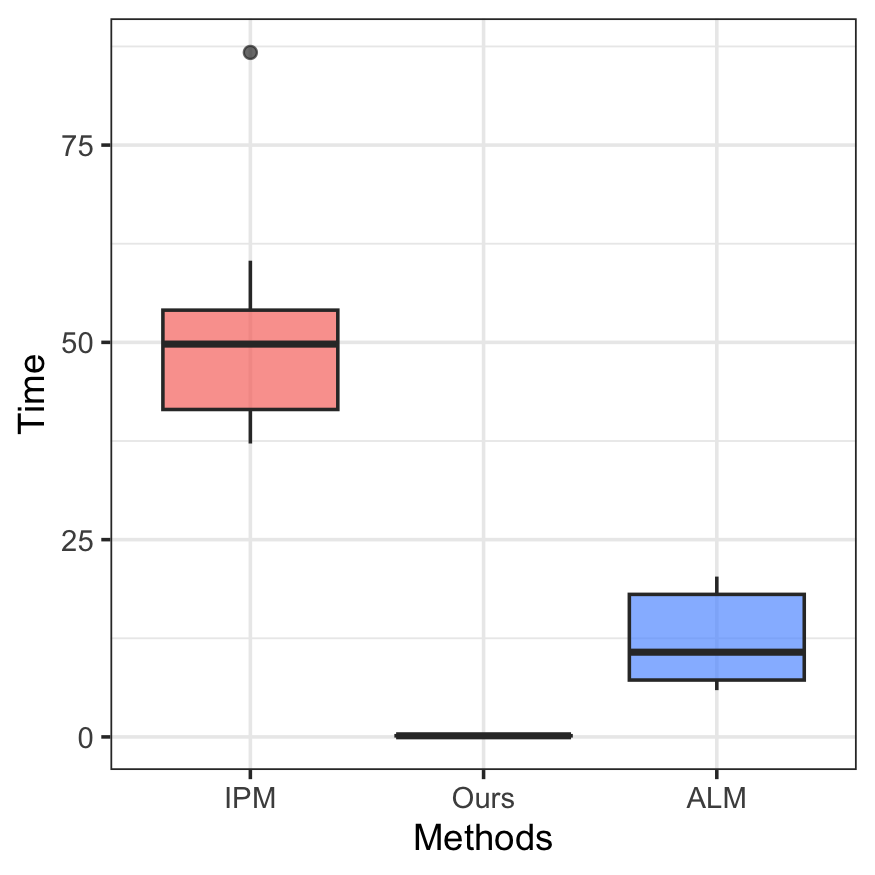}
    \hfill
    \includegraphics[width=0.32\textwidth]{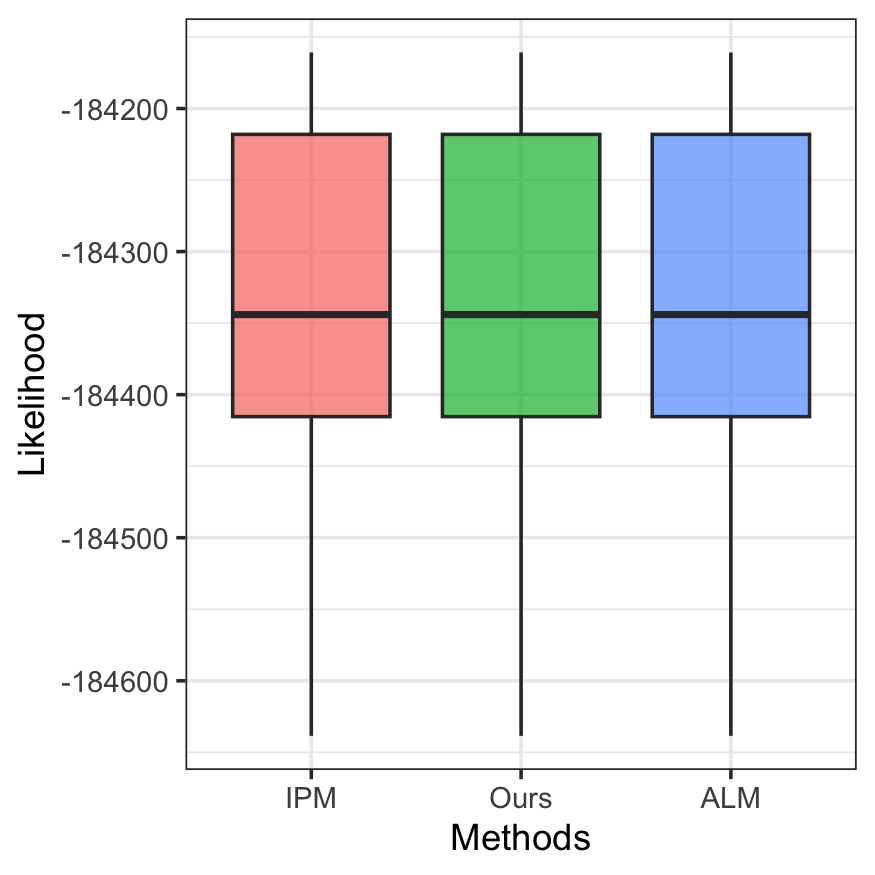}
    \hfill
    \includegraphics[width=0.32\textwidth]{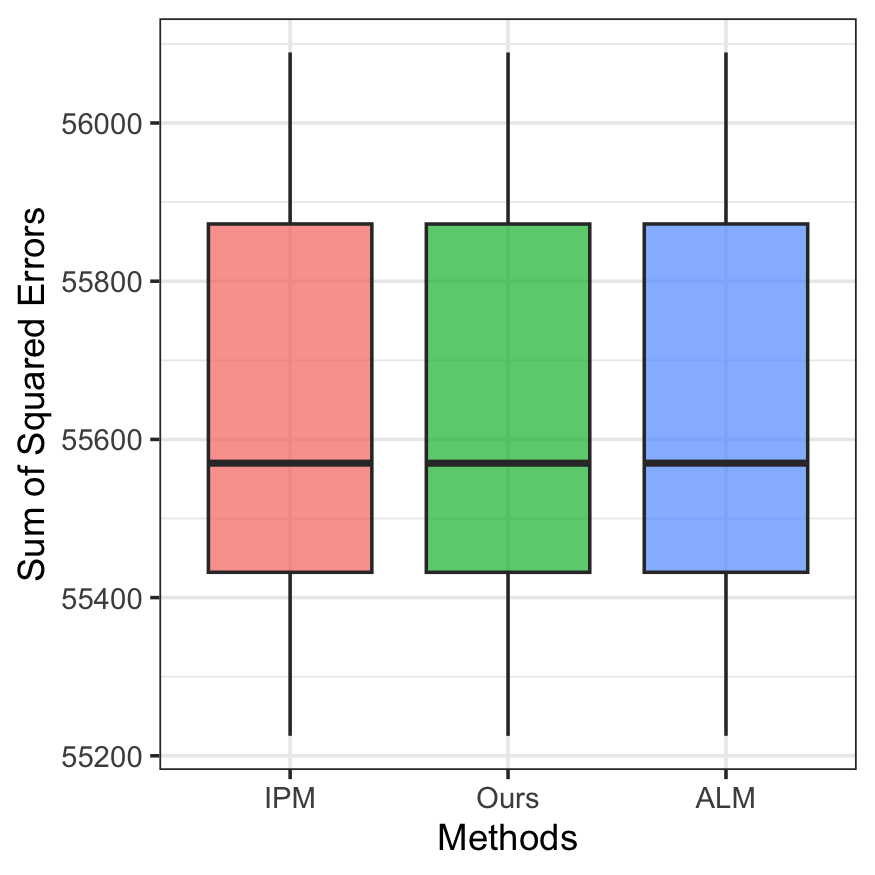}
    \caption{Comparison of IPM, Ours, and ALM in terms of computation time, log-likelihood, and sum of squared errors in empirical Bayes estimation.}
    \label{fig:comparison}
\end{figure}

\subsection{Bernstein's theorem}

It was first shown in §61 of \citet{bernstein1912ordre} that analytic functions admit polynomial approximations with exponential convergence. A modern treatment appears as Theorem 8.2 of \citet{trefethen2019approximation}, which we restate below. 

\begin{theorem}\label{thm:Bernstein}
Fix $B>0$ and let $h:[-B,B]\to\mathbb{R}$ admit an analytic continuation to the Bernstein ellipse
\[
E^{(B)}_\rho \coloneqq \bigg\{ z \in \mathbb{C} : 
z=\frac{B}{2}\Big( w + \frac{1}{w} \Big),\ \frac{1}{\rho}< |w| < \rho \bigg\},
\qquad \rho>1.
\]
Then for any $J\in\mathbb{N}$, the degree-$J$ Chebyshev projection of $h$
\[
\mathrm{Che}_J[h]\coloneqq \sum_{k=0}^J c_k\, T_k
\]
satisfies $|c_k|\le 2C_h/\rho^k$, where
\[
C_h\coloneqq \sup_{z\in E^{(B)}_\rho}|h(z)|,
\]
and $T_k$ denotes the (scaled) Chebyshev polynomial defined by
\begin{equation}\label{eq:Chebyshev}
T_k(B\cos\theta)\coloneqq \cos(k\theta).
\end{equation}
Moreover,
\begin{equation}\label{eq:Bernstein}
\|h-\mathrm{Che}_J[h]\|_{\infty,B}
\le \frac{2C_h}{(\rho-1)\rho^{J}},
\end{equation}
where $\|\cdot\|_{\infty,B}$ denotes the supremum norm on $[-B,B]$.
\end{theorem}

Theorem~\ref{thm:Bernstein} extends to the multivariate setting; see \citet{trefethen2017multivariate} and \citet{wang2020analysis}. We record a convenient formulation and include a brief proof for completeness.

\begin{theorem}\label{thm:Bernstein-multi}
Fix $B \coloneqq (B_1,\dots,B_d)$ with $B_l>0$ for all $l\in[d]$. Let
\[
h:\prod_{l=1}^d [-B_l,B_l]\to\mathbb R
\]
admit an analytic continuation to the Bernstein polyellipse
\[
\mathcal E_\rho^{(B)} \coloneqq \prod_{l=1}^d E_{\rho_l}^{(B_l)},
\]
where $\rho\coloneqq(\rho_1,\dots,\rho_d)$ and $\rho_l>1$ for all $l\in[d]$. Then, for every $J\in\mathbb N$, there exists a polynomial $\mathrm{poly}_J[h]$ of total degree at most $J$ such that
\begin{equation}\label{eq:Bernstein-multi}
\|h-\mathrm{poly}_J[h]\|_{\infty,B}
\le
2^d
\Big( \sup_{z\in\mathcal E_\rho^{(B)}} |h(z)| \Big)
\Big( \prod_{l=1}^d \frac{\rho_l}{\rho_l-1} \Big)
\Big( \sum_{l=1}^d \frac{1}{\rho_l^{\lfloor J/d\rfloor+1}} \Big),
\end{equation}
where $\|\cdot\|_{\infty,B}$ denotes the supremum norm on $\prod_{l=1}^d [-B_l,B_l]$, and \(\lfloor x\rfloor\) denotes the largest integer less than or equal to \(x\in\mathbb R\).
\end{theorem}

\begin{proof}
By rescaling each coordinate, we may assume without loss of generality that
$
B_1=\cdots=B_d=1.
$
Define
$$
\tilde h(w)\coloneqq h(z),
\qquad
z_l=\frac{1}{2}\Big(w_l+\frac{1}{w_l}\Big), \quad l\in[d],
$$
and consider $\tilde h$ on the polyannulus
$$
\mathcal A\coloneqq \prod_{l=1}^d \Bigl\{ w_l\in\mathbb C:\frac{1}{\rho_l}<|w_l|<\rho_l \Bigr\}.
$$
Since $\tilde h$ is analytic on $\mathcal A$, Theorem 2 on p.~35 of \cite{shabat1992introduction} yields that it admits a Laurent series expansion
$$
\tilde h(w)=\sum_{\nu\in\mathbb Z^d} a_\nu w^\nu,
$$
where
$$
a_\nu=\frac{1}{(2\pi i)^d}
\int_{\prod_{l=1}^d\{u_l\in\mathbb C:|u_l|=r_l\}}
\frac{\tilde h(u)}{u^{\nu+1}}\,du
$$
for any $r_l\in(1,\rho_l)$, $l\in[d]$. 

Now $\tilde h$ has the symmetry
$$
\tilde h(w)=\tilde h(\tilde w)
\qquad\text{whenever}\qquad
\tilde w_l\in\{w_l,w_l^{-1}\}
\ \text{for all } l\in[d].
$$
Hence the Laurent coefficients satisfy
$
a_\nu=a_{\nu'}
$
whenever $\nu'$ is obtained from $\nu$ by changing the sign of any subset of coordinates. Therefore, regrouping the Laurent series gives
$$
\tilde h(w)=\sum_{\alpha\in\mathbb N^d} c_\alpha \tilde q_\alpha(w),
$$
where
$$
\tilde q_\alpha(w)\coloneqq \prod_{l=1}^d \frac{w_l^{\alpha_l}+w_l^{-\alpha_l}}{2},
\qquad
c_\alpha\coloneqq 2^{s(\alpha)} a_\alpha,
$$
and
$
s(\alpha)\coloneqq \#\{l\in[d]:\alpha_l>0\}.
$

Now define
$$
q_\alpha(z)\coloneqq \tilde q_\alpha(w),
$$
where $w$ is any vector satisfying $z_l=(w_l+w_l^{-1})/2$ for each $l\in[d]$. This is well defined because of the symmetry of $\tilde q_\alpha$. Hence, 
$$
h(z)=\sum_{\alpha\in\mathbb N^d} c_\alpha q_\alpha(z).
$$
For $x\in[-1,1]^d$, 
$$
q_\alpha(x)=\prod_{l=1}^d T_{\alpha_l}(x_l),
$$
where $T_k$ is the Chebyshev polynomial \eqref{eq:Chebyshev}; see, for example, equation (3.8) in \citet{trefethen2019approximation}. So, 
$$
\|q_{\alpha}\|_{\infty,B}\le1. 
$$

Let
$$
C_h\coloneqq \sup_{z\in\mathcal E_\rho^{(B)}} |h(z)|.
$$
From the integral representation, for any $\alpha\in\bN^d$, 
$$
|c_\alpha|
=
2^{s(\alpha)} |a_\alpha|
\le
\frac{2^d C_h}{r^\alpha}.
$$

Combining these bounds, for any $J\in\mathbb N$, we have
$$
\begin{aligned}
\Big\|h-\sum_{|\alpha|\le J} c_\alpha q_\alpha\Big\|_{\infty,B}
&\le
\sum_{|\alpha|>J} |c_\alpha|\\
&\le
\sum_{l=1}^d \sum_{\alpha_l>\lfloor J/d\rfloor} |c_\alpha|\\
&\le
2^d C_h
\sum_{l=1}^d
\sum_{\alpha_l>\lfloor J/d\rfloor} \frac{1}{r^\alpha}\\
&=
2^d C_h
\Bigl( \prod_{l=1}^d \frac{r_l}{r_l-1} \Bigr)
\Bigl( \sum_{l=1}^d \frac{1}{r_l^{\lfloor J/d\rfloor+1}} \Bigr).
\end{aligned}
$$
Letting $r_l\uparrow \rho_l$ for all $l\in[d]$, we obtain \eqref{eq:Bernstein-multi}.
\end{proof}

\subsection{Proof of main results in Section~\ref{sec:data-compression}}

Under \ref{(EF)}, 
for any \(g \in \mathcal{G}\) with compact support, it follows from Morera's theorem (Theorem 5.1 of \citet{stein2010complex}) that the likelihood ratio function \(l_g(x)\) extends to an entire function. In addition, \(\log l_g(x)\) admits an analytic extension to a complex neighborhood of the real line. For \(z = x + iy \in \mathbb{C}\), we denote by \(\Re z \coloneqq x\) and \(\Im z \coloneqq y\) its real and imaginary parts.

\begin{lemma}\label{lem:analytic-log-density}
Assume \ref{(EF)}. Take some $g\in\cG$ and assume $M_g\coloneqq\sup_{\theta\in\supp(g)}|\theta|<\infty$. For any $B>0$, set
$$
\rho=\frac{\pi}{4BM_g}+\sqrt{1+\Big(\frac{\pi}{4BM_g}\Big)^2}. 
$$
Then $\log l_g(x)$ extends analytically to the Bernstein ellipse $E_{\rho}^{(B)}$. Moreover, 
$$
\sup_{z\in E_\rho^{(B)}}|\log l_g(z)|\le CBM_g+\sup_{|\theta|\le M_g}|\kappa(\theta)|,
$$
for a universal constant $C>0$, provided $BM_g\ge1$. 
\end{lemma}

\begin{proof}
For $z\in\bC$, we write
$$
\begin{aligned}
l_g(z)&=\int\exp\big(\theta z-\kappa(\theta)\big)\dif g(\theta)\\
&=\int\exp\big(\Re z\theta-\kappa(\theta)\big)\big\{\cos(\Im z\theta)+i\sin(\Im z\theta)\big\}\dif g(\theta),
\end{aligned}
$$
and we have $\Re[l_g(z)]>0$ whenever $|\Im z|<\pi/(2M_g)$. So, taking the principal branch of the logarithm, the function $\log l_g(z)$ is holomorphic in $E_{\rho}^{(B)}$. 

Fix an arbitrary $z\in E_{\rho}^{(B)}$, we notice that
$$
\exp\Big(-|\Re z|M_g-\sup_{|\theta|\le M_g}|\kappa(\theta)|\Big)\cos\big(|\Im z|M_g\big)\le|l_g(z)|\le \exp\Big(|\Re z|M_g+\sup_{|\theta|\le M_g}|\kappa(\theta)|\Big). 
$$
Therefore, 
$$
\begin{aligned}
|\log l_g(z)|&\le |\Re z|M_g+\sup_{|\theta|\le M_g}|\kappa(\theta)|+\log\frac{1}{\cos\big(|\Im z|M_g\big)}+\frac{\pi}{2}\\
&\le BM_g\sqrt{1+\left(\frac{\pi}{4BM_g}\right)^2}+\sup_{|\theta|\le M_g}|\kappa(\theta)|+\frac{\log2}{2}+\frac{\pi}{2}\\
&\le CBM_g+\sup_{|\theta|\le M_g}|\kappa(\theta)|. 
\end{aligned}
$$
This completes the proof. 
\end{proof}

Because the exact estimator $\hat{g}_{J_n}$ in \eqref{eq:hatgJn} is often intractable due to the infinite-dimensional nature of the set $\cG$, an approximate version is needed. Therefore, assuming the upper bound in \eqref{eq:bound} is available, we restrict our focus to approximate estimators
\begin{equation}\label{eq:tildegJn}
\tilde{g}_{J_n}\in\Big\{g\in\cG:\supp(g)\subseteq[-M_n,M_n],\int\log l_g(x)\dif P_{J_n}(x)\ge\int\log l_{\hat{g}_{J_n}}(x)\dif P_{J_n}(x)-\varepsilon_n\Big\}, 
\end{equation}
where $\{\varepsilon_n\}_{n\ge1}$ is a sequence of tolerance levels. Accordingly, we present a theorem that accounts for this additional approximation error, yielding Theorem~\ref{thm:approximateNPMLE1} as a special case when $\varepsilon_n\equiv0$. 

\begin{theorem}\label{thm:approximateNPMLE2}
Assume \ref{(EF)}. Let \(\{M_n\}_{n\ge1}\) be a sequence of positive numbers such that \eqref{eq:bound} holds. Assume that \(|X|_n M_n \ge 1\). Then the likelihood criterion \eqref{eq:LRT} holds for $\tilde{g}=\tilde{g}_{J_n}$ provided that $\varepsilon_n\le\Delta_n/(2n)$, and
\[
J_n
=
\bigg\lceil
2|X|_nM_n
\log\Big(
\frac{Cn|X|_nM_n(|X|_nM_n+\sup_{|\theta|\le M_n}|\kappa(\theta)|)}{\Delta_n}
\Big)
\bigg\rceil,
\]
where \(C>0\) is a universal constant.
\end{theorem}

\begin{proof}
Theorem 3.3.1 of \citet{szeg1939orthogonal} implies that $\mathrm{supp}(P_{J_n})\subseteq\mathrm{conv}\big(\mathrm{supp}(P_n)\big)\subseteq[-|X|_n,|X|_n]$. 
We therefore have
$$
\begin{aligned}
\ell_n(f_{\hat{g}})-\ell_n(f_{\tilde{g}_{J_n}})&=n\int\big(\log l_{\hat g}(x)-\log l_{\tilde{g}_{J_n}}(x)\big)\dif P_n(x)\\
&\le n\int\log l_{\hat g}(x)\dif(P_n-P_{J_n})(x)+n\int \log l_{\tilde g_{J_n}}(x)\dif (P_{J_n}-P_n)(x)\\
&+n\int\big(\log l_{\hat g_{J_n}}(x)- \log l_{\tilde g_{J_n}}(x)\big)\dif P_{J_n}(x)\\
&\le2n\max_{g\in\{\hat{g},\tilde{g}_{J_n}\}}\Big|\int\log l_g(x)\dif(P_n-P_{J_n})(x)\Big|+\frac{\Delta_n}{2}\\
&\le4n\max_{g\in\{\hat{g},\tilde{g}_{J_n}\}}\max_{P\in\{P_n,P_{J_n}\}}\Big|\int\big(\log l_g(x)-\mathrm{Che}_{J_n}[\log l_g](x)\big)\dif P(x)\Big|+\frac{\Delta_n}{2}\\
&\le4n\max_{g\in\{\hat{g},\tilde{g}_{J_n}\}}\|\log l_g-\mathrm{Che}_{J_n}[\log l_g]\|_{\infty,B}+\frac{\Delta_n}{2}, 
\end{aligned}
$$
where $\mathrm{Che}_J[\cdot]$ and $\|\cdot\|_{\infty,B}$ are as in Theorem~\ref{thm:Bernstein} with $B=|X|_n$. Invoking the bound \eqref{eq:Bernstein} together with the coefficients from Lemma~\ref{lem:analytic-log-density} yields
$$
\begin{aligned}
\max_{g\in\{\hat{g},\tilde{g}_{J_n}\}}\|\log l_g-\mathrm{Che}_{J_n}[\log l_g]\|_{\infty,B}&\le\frac{8BM_n(CBM_n+\sup_{|\theta|\le M_n}|\kappa(\theta)|)}{\pi\big(1+\pi/(4BM_n)\big)^{J_n}}\\
&\le\frac{8BM_n(CBM_n+\sup_{|\theta|\le M_n}|\kappa(\theta)|)}{\pi 2^{(\pi J_n)/(4 BM_n)}}, 
\end{aligned}
$$
where we used $(1+1/x)^x\ge2$ for $x\ge1$. Thus, to have \eqref{eq:LRT} for $\tilde{g}=\tilde{g}_{J_n}$, we will need
$$
J_n\ge2BM_n\log\Big(\frac{64nBM_n(CBM_n+\sup_{|\theta|\le M_n}|\kappa(\theta)|)}{\pi\Delta_n}\Big). 
$$
A redefinition of the constant $C$ yields the desired result. 
\end{proof}

\bigskip

\begin{proof}[Proof of Corollary~\ref{cor:order}]
Under \ref{(S1)}, Lemma~\ref{lem:tail} gives
\[
\mathbb{P}(|X|_n \ge t) \le n \exp(-b_1 t), \qquad t \ge b_0,
\]
for some constants $b_0,b_1>0$. Choosing
\[
t=\frac{2}{b_1}\log n,
\]
we obtain
\[
\mathbb{P}\Big(|X|_n \ge \frac{2}{b_1}\log n\Big)\le \frac{1}{n}.
\]
Hence,
\[
|X|_n = O_{\mathbb{P}}(\log n).
\]
The asserted order of $J_n$ holds with $M_n=M$.

Under \ref{(LT)}, the same argument shows that
\[
|X|_n = O_{\mathbb{P}}\bigl((\log n)^{1/\beta_0}\bigr).
\]
If, in addition, \ref{(S2)} holds, then Lemma~\ref{lem:scenario2} implies that one may choose
\[
M_n = O(|X|_n^{\alpha_1})
    = O_{\mathbb{P}}\bigl((\log n)^{\alpha_1/\beta_0}\bigr).
\]
Consequently,
\[
\sup_{|\theta|\le M_n} |\kappa(\theta)|
    = O\bigl(M_n^{1+1/\alpha_2}\bigr)
    = O_{\mathbb{P}}\Big((\log n)^{\frac{\alpha_1}{\beta_0}\big(1+\frac{1}{\alpha_2}\big)}\Big).
\]
These yield the claimed order of $J_n$.
\end{proof}

\subsection{Proof of main results in Section~\ref{sec:marginal-density}}

We first recall the $\varepsilon$-covering number of a general set $\mathcal{S}$ with respect to a semimetric $\rho$:
\[
N(\varepsilon,\mathcal{S},\rho)
\coloneqq
\inf\Bigl\{
N:\exists\, s_1,\dots,s_N\in\mathcal{S}\ \text{such that}\ 
\mathcal{S}\subseteq\bigcup_{j=1}^N B_\rho(s_j,\varepsilon)
\Bigr\},
\]
where
\[
B_\rho(s,\varepsilon)\coloneqq\{s'\in\mathcal{S}:\rho(s,s')<\varepsilon\}.
\]

\begin{lemma}\label{lem:entropy0}
Recall the class $\cG$ in \eqref{eq:mixing} with $\Theta_0$ from \ref{(EF)}. Assume that $M<\infty$, and let $\{T_k\}_{k\in\mathbb{N}}$ denote the Chebyshev polynomials defined in \eqref{eq:Chebyshev} with $B=M$. For $g_1,g_2\in\mathcal{G}$, define
\[
\Delta_J(g_1,g_2)
\coloneqq
\max_{k\in[J]}
\left|
\int T_k(\theta)\,\mathrm{d}(g_1-g_2)(\theta)
\right|.
\]
Then, for every $\varepsilon\in(0,1/2)$,
\[
\log N(\varepsilon,\mathcal{G},\Delta_J)
\le
C J \log(1/\varepsilon),
\]
where $C>0$ is a universal constant.
\end{lemma}

\begin{proof}
Simply note that $|T_k(\theta)|\le 1$ for all $\theta\in[-M,M]$. Hence the map
\[
g\mapsto \Big(\int T_k(\theta)\,\mathrm{d}g(\theta)\Big)_{k\in[J]}
\]
embeds $\mathcal{G}$ into $[-1,1]^J$, and $\Delta_J$ is just the induced supremum distance $\|\cdot\|_\infty$. Therefore,
\[
N(\varepsilon,\mathcal{G},\Delta_J)\le N(\varepsilon,[-1,1]^J,\|\cdot\|_{\infty})\le \Big(\frac{C}{\varepsilon}\Big)^J,
\]
for some $C>0$, which gives the claim.
\end{proof}

\begin{lemma}\label{lem:quantity1}
Under \ref{(S1)}, there exists $\rho_0>1$ such that
\[
\sup_{x\in\mathbb{R}} \sup_{\theta\in E_{\rho_0}^{(M)}} |\bar{p}_\theta(x)| < \infty.
\]
\end{lemma}

\begin{proof}
Under scenario~\ref{(S1)}, the moment generating function $Z(\theta)$ admits an analytic extension to the Bernstein ellipse $E_{\rho_0}^{(M)}$ for any
\[
\rho_0 \in \left(1,\, 1+\frac{\tau}{M}\right).
\]
Since $Z(\theta)>0$ for all $\theta\in[-M,M]$, by continuity we may choose $\rho_0>1$ sufficiently close to $1$ so that
\[
\inf_{\theta\in E_{\rho_0}^{(M)}} |Z(\theta)| > 0.
\]
Therefore,
\[
\sup_{x\in\mathbb{R}} \sup_{\theta\in E_{\rho_0}^{(M)}} |\bar{p}_\theta(x)|
\le
\sup_{x\in\mathbb{R}} \sup_{\theta\in E_{\rho_0}^{(M)}}
\frac{\exp\big( |\Re(\theta)|\,|x| - (M+\tau)|x| \big)}{|Z(\theta)|}
< \infty.
\]
This proves the claim.
\end{proof}

\begin{lemma}\label{lem:entropy}
Suppose that scenario~\ref{(S1)} holds. Then there exists a constant
$C_{p_0,M} > 0$, depending only on $p_0$ and $M$, such that for every
$\varepsilon \in (0,1/2)$,
\[
\log N(\varepsilon,\bar{\mathcal F},\|\cdot\|_\infty)
\le
C_{p_0,M}\,\bigl[\log(1/\varepsilon)\bigr]^2.
\]
Here $\|\cdot\|_\infty$ denotes the supremum norm on $\mathbb R$.
\end{lemma}

\begin{proof}
Fix an arbitrary $x\in\mathbb{R}$. Apply Theorem~\ref{thm:Bernstein} with $h(\cdot)=\bar{p}_{\cdot}(x)$, $B=M$, and $\rho=\rho_0$, where $\rho_0$ is given by Lemma~\ref{lem:quantity1}. Then, for any $g_1,g_2\in\mathcal{G}$ and any $J\in\mathbb{N}$,
\[
\begin{aligned}
|\bar{f}_{g_1}(x)-\bar{f}_{g_2}(x)|
&=\Big|\int \bar{p}_{\theta}(x)\,\mathrm{d}(g_1-g_2)(\theta)\Big|\\
&\le \Big|\int \mathrm{Che}_J[h](\theta)\,\mathrm{d}(g_1-g_2)(\theta)\Big|
   + \frac{4C_h}{(\rho_0-1)\rho_0^J}\\
&\le \sum_{k=1}^J \frac{2C_h}{\rho_0^k}
   \Big|\int T_k(\theta)\,\mathrm{d}(g_1-g_2)(\theta)\Big|
   + \frac{4C_h}{(\rho_0-1)\rho_0^J}\\
&\le \frac{2C_h}{\rho_0-1}
\Big(\Delta_J(g_1,g_2)+\frac{2}{\rho_0^J}\Big),
\end{aligned}
\]
where $\Delta_J$ is defined as in Lemma~\ref{lem:entropy0}. Therefore,
\[
\|\bar{f}_{g_1}-\bar{f}_{g_2}\|_\infty
\le
\frac{2C_{p_0,M}}{\rho_0-1}
\Big(\Delta_J(g_1,g_2)+\frac{2}{\rho_0^J}\Big),
\]
where
\[
C_{p_0,M}\coloneqq \sup_{x\in\mathbb{R}}\sup_{\theta\in E_{\rho_0}^{(M)}} |\bar{p}_\theta(x)|.
\]

For any $\varepsilon\in(0,1/2)$, to make the second term bounded by $\varepsilon/2$, we choose
\[
J=\Big\lceil\frac{\log(1/\varepsilon)+\log(8C_{p_0,M})-\log(\rho_0-1)}{\log\rho_0}\Big\rceil.
\]
Then, by Lemma~\ref{lem:entropy0},
\[
\log N(\varepsilon,\bar{\mathcal{F}},\|\cdot\|_\infty)
\le
C J \log\Big(\frac{4C_{p_0,M}}{(\rho_0-1)\varepsilon}\Big)
\le
C_{p_0,M}\big[\log(1/\varepsilon)\big]^2,
\]
where $C_{p_0,M}$ is redefined in the last step. This completes the proof.
\end{proof}

\bigskip

\begin{proof}[Proof of Theorem~\ref{thm:density-rate1}]
Fix $n\in\bN$ and $\varepsilon,\delta>0$. Let $\{\bar{f}_k\}_{k\in[N]}$ be an $\varepsilon$-covering of $\bar{\cF}$ given by Lemma~\ref{lem:entropy}. Define $\mathcal{N}\subseteq [N]$ to be the set of indices $k$ for which there exists some $\bar{f}_{k,0}\in\bar{\cF}$ such that
\[
\|\bar{f}_{k,0}-\bar{f}_k\|_{\infty}\le\varepsilon,
\qquad
H^2(f_{k,0},f_{g_0})\ge\delta,
\]
where $f_{k,0}$ denotes the marginal density corresponding to $\bar{f}_{k,0}$.

If
\[
H^2(f_{\tilde{g}},f_{g_0})\ge\delta,
\]
then, $\mathcal{N}\neq\emptyset$, and
\[
\exp\big(-A(\log n)^2\big)
\le
\prod_{i=1}^n\frac{f_{\tilde{g}}(X_i)}{f_{g_0}(X_i)}
=
\prod_{i=1}^n\frac{\bar{f}_{\tilde{g}}(X_i)}{\bar{f}_{g_0}(X_i)}
\le
\sup_{k\in\mathcal{N}}\prod_{i=1}^n\frac{\bar{f}_{k,0}(X_i)+2\varepsilon}{\bar{f}_{g_0}(X_i)}.
\]
Consequently,
\[
\begin{aligned}
\bP\big(H^2(f_{\tilde{g}},f_{g_0})\ge\delta\big)
&\le \bP\Big(\exp\big(-A(\log n)^2\big)\le \sup_{k\in\mathcal{N}}\prod_{i=1}^n\frac{\bar{f}_{k,0}(X_i)+2\varepsilon}{\bar{f}_{g_0}(X_i)}\Big)\\
&\le N\exp\Big(\frac{A}{2}(\log n)^2\Big)\sup_{k\in\mathcal{N}}\prod_{i=1}^n\bE\sqrt{\frac{\bar{f}_{k,0}(X_i)+2\varepsilon}{\bar{f}_{g_0}(X_i)}}\\
&\le N\exp\Big(\frac{A}{2}(\log n)^2\Big)\sup_{k\in\mathcal{N}}\exp\bigg[n\int\Big(\sqrt{\frac{\bar{f}_{k,0}+2\varepsilon}{\bar{f}_{g_0}}}-1\Big)f_{g_0}\dif \mu\bigg]\\
&\le \exp\Big(C_{p_0,M}\big[\log(1/\varepsilon)\big]^2+\frac{A}{2}(\log n)^2-\frac{n\delta}{2}+n\sqrt{2\varepsilon}\,I_0\Big),
\end{aligned}
\]
where in the third inequality we used $x\le \exp(x-1)$. To justify the last inequality, note that
\[
\begin{aligned}
\int\Big(\sqrt{\frac{\bar{f}_{k,0}+2\varepsilon}{\bar{f}_{g_0}}}-1\Big)f_{g_0}\dif \mu
&\le
\int\Big(\sqrt{\frac{f_{k,0}}{f_{g_0}}}-1\Big)f_{g_0}\dif \mu
+\sqrt{2\varepsilon}\int\sqrt{\frac{1}{\bar{f}_{g_0}}}f_{g_0}\dif \mu\\
&\le
-\frac{H^2(f_{k,0},f_{g_0})}{2}+\sqrt{2\varepsilon}\,I_0,
\end{aligned}
\]
where
\[
I_0\coloneqq \int \exp\Big(\big(M+\frac{\tau}{2}\big)|x|-\frac{1}{2}\inf_{|\theta|\le M}\kappa(\theta)\Big)p_0(x)\dif \mu(x)<\infty.
\]

Finally, take $\varepsilon=1/n^2$ and
\[
\delta=tC_{p_0,M,A}\frac{(\log n)^2}{n},
\]
with $C_{p_0,M,A}>0$ sufficiently large. This completes the proof.
\end{proof}

Theorem~\ref{thm:Bernstein} directly yields an entropy bound for a class of analytic functions. An analogous bound for the multivariate case is given in Theorem 2.7.16 of \citet{van1996weak}. 

\begin{lemma}\label{lem:entropy-analytic}
Fix $B>0$ and $\rho>1$. Let $\mathcal{H}$ be a class of functions
$h:[-B,B]\to \mathbb{R}$, each of which admits an analytic continuation to
the Bernstein ellipse $E_\rho^{(B)}$. Suppose that
\[
C_{\mathcal{H}}
\coloneqq
\sup_{h\in\mathcal{H}}\sup_{z\in E_\rho^{(B)}} |h(z)| < \infty .
\]
Then, for every $\varepsilon\in(0,1/2)$,
\[
\log N(\varepsilon,\mathcal{H},\|\cdot\|_{\infty,B})
\le
C_\rho\bigl[\log(1/\varepsilon)+\log C_{\mathcal{H}}\bigr]^2,
\]
where $C_\rho$ is a constant depending only on $\rho$.
\end{lemma}

\begin{proof}
Let $h_1,h_2\in\mathcal{H}$ and $J\in\mathbb{N}$. By
Theorem~\ref{thm:Bernstein},
\[
\begin{aligned}
\|h_1-h_2\|_{\infty,B}
&\le
\|\mathrm{Che}_J[h_1]-\mathrm{Che}_J[h_2]\|_{\infty,B}
+
2\max_{h\in\mathcal{H}}\|h-\mathrm{Che}_J[h]\|_{\infty,B} \\
&\le
\sum_{k=0}^J |c_{k,1}-c_{k,2}|
+
\frac{4C_{\mathcal{H}}}{(\rho-1)\rho^J} \\
&\le
\frac{2C_{\mathcal{H}}\rho}{\rho-1}
\max_{0\le k\le J}
\Big(
\frac{\rho^k}{2C_{\mathcal{H}}}|c_{k,1}-c_{k,2}|
\Big)
+
\frac{4C_{\mathcal{H}}}{(\rho-1)\rho^J},
\end{aligned}
\]
where $c_{k,i}$ denotes the $k$th coefficient of
$\mathrm{Che}_J[h_i]$, $i=1,2$.

Choose
\[
J
=
\left\lceil
\frac{
\log(1/\varepsilon)+\log(8C_{\mathcal{H}})-\log(\rho-1)
}{
\log \rho
}
\right\rceil ,
\]
so that the approximation error term is bounded by $\varepsilon/2$. An
argument analogous to the proof of Lemma~\ref{lem:entropy} then gives
\[
\log N(\varepsilon,\mathcal{H},\|\cdot\|_{\infty,B})
\le
CJ\log\Big(
\frac{4C_{\mathcal{H}}\rho}{(\rho-1)\varepsilon}
\Big)
\le
C_\rho\bigl[\log(1/\varepsilon)+\log C_{\mathcal{H}}\bigr]^2 .
\]
This proves the claim.
\end{proof}

\begin{lemma}\label{lem:quantity2}
Assume \ref{(S2)}. Fix arbitrary $g\in\mathcal{G}$, $B\ge 1$, and
$\rho>1$. Then
\[
\sup_{z\in E_{\rho}^{(B)}} |l_g(z)|
\le
\exp\bigl(C_{p_0,\rho} B^{1+\alpha_1}\bigr),
\]
where $C_{p_0,\rho}$ is a constant depending only on $p_0$ and $\rho$.
\end{lemma}

\begin{proof}
Under scenario~\ref{(S2)}, for any $z\in E_{\rho}^{(B)}$,
\[
\begin{aligned}
|l_g(z)|
&\le
\int \exp\bigl(|\Re z|\,|\theta|-\kappa(\theta)\bigr)\, \dif g(\theta) \\
&\le
\sup_{\theta\in\mathbb{R}}
\exp\bigl(|\Re z|\,|\theta|-\kappa(\theta)\bigr) \\
&\le
\exp(|\Re z|c_0)
+
\sup_{|\theta|\ge c_0}
\exp\bigl(|\Re z|\,|\theta|-\kappa(\theta)\bigr).
\end{aligned}
\]
Since $z\in E_{\rho}^{(B)}$, we have $|\Re z|\le \rho B$. Therefore,
using \ref{(S2)} and $B\ge 1$,
\[
|l_g(z)|
\le
\exp(c_0\rho B)
+
\exp\bigl(C_{p_0,\rho} B^{1+\alpha_1}\bigr)
\le
\exp\bigl(C_{p_0,\rho} B^{1+\alpha_1}\bigr),
\]
after enlarging $C_{p_0,\rho}$ if necessary. Taking the supremum over
$z\in E_{\rho}^{(B)}$ proves the claim.
\end{proof}

\begin{lemma}\label{lem:entropy2}
Suppose that scenario~\ref{(S2)} holds, and let $B \ge 1$. Then there
exists a constant $C_{p_0}>0$, depending only on $p_0$, such that, for
every $\varepsilon\in(0,1/2)$,
\[
\log N(\varepsilon,\mathcal L,\|\cdot\|_{\infty,B})
\le
C_{p_0}\bigl[\log(1/\varepsilon)+B^{1+\alpha_1}\bigr]^2 .
\]
\end{lemma}

\begin{proof}
Apply Lemma~\ref{lem:entropy-analytic} with $\mathcal{H}=\mathcal{L}$ and
$\rho=2$. This gives
\[
\log N(\varepsilon,\mathcal L,\|\cdot\|_{\infty,B})
\le
C_2\bigl[\log(1/\varepsilon)+\log C_{\mathcal L}\bigr]^2 .
\]
By Lemma~\ref{lem:quantity2},
\[
C_{\mathcal L}
=
\sup_{l\in\mathcal L}\sup_{z\in E_2^{(B)}} |l(z)|
\le
\exp\bigl(C_{p_0,2}B^{1+\alpha_1}\bigr).
\]
Therefore,
\[
\log N(\varepsilon,\mathcal L,\|\cdot\|_{\infty,B})
\le
C_2\bigl[\log(1/\varepsilon)+C_{p_0,2}B^{1+\alpha_1}\bigr]^2.
\]
This completes the proof.
\end{proof}

\bigskip

\begin{proof}[Proof of Theorem~\ref{thm:density-rate2}]
Fix $n\in\mathbb{N}$, $\varepsilon,\delta>0$, and $B\ge 1$. Let $\{l_k\}_{k\in[N]}$ be an $\varepsilon$-covering of $\mathcal{L}$ from Lemma~\ref{lem:entropy2}. Define $\mathcal{N}\subseteq [N]$ to be the set of indices $k$ for which there exists some $l_{k,0}\in\mathcal{L}$ such that
\[
\|l_{k,0}-l_k\|_{\infty,B}\le \varepsilon,
\qquad
H^2(f_{k,0},f_{g_0})\ge \delta,
\]
where $f_{k,0}$ denotes the marginal density corresponding to $l_{k,0}$.

If
\[
H^2(f_{\tilde{g}},f_{g_0})\ge \delta,
\qquad\text{and}\qquad
|X|_n\le B,
\]
then $\mathcal{N}\neq\emptyset$, and
\[
\exp\big(-A(\log n)^{\gamma_0}\big)
\le
\prod_{i=1}^n\frac{f_{\tilde{g}}(X_i)}{f_{g_0}(X_i)}
=
\prod_{i=1}^n\frac{l_{\tilde{g}}(X_i)}{l_{g_0}(X_i)}
\le
\sup_{k\in\mathcal{N}}\prod_{i=1}^n\frac{l_{k,0}(X_i)+2\varepsilon}{l_{g_0}(X_i)}.
\]
Consequently, using the same argument as in the proof of Theorem~\ref{thm:density-rate1},
\[
\begin{aligned}
&\bP\big(H^2(f_{\tilde{g}},f_{g_0})\ge\delta\big)
-\bP\big(|X|_n\ge B\big)\\
&\le
\bP\Big(\exp\big(-A(\log n)^{\gamma_0}\big)\le \sup_{k\in\mathcal{N}}\prod_{i=1}^n\frac{l_{k,0}(X_i)+2\varepsilon}{l_{g_0}(X_i)}\Big)\\
&\le
\exp\Big(\log N+\frac{A}{2}(\log n)^{\gamma_0}-\frac{n\delta}{2}+n\sqrt{2\varepsilon}\Big).
\end{aligned}
\]

By~\ref{(LT)}, for all sufficiently large $n$,
\[
\bP\big(|X|_n\ge B\big)
\le
\exp\big(-t^{\beta_0/[2(1+\alpha_1)]}\log n\big),
\]
provided that we choose
\[
B=t^{1/[2(1+\alpha_1)]}\Big(\frac{2}{b_1}\log n\Big)^{1/\beta_0}.
\]

With $\varepsilon=1/n^{2}$, it follows that
\[
\log N
\le
C_{p_0}\big[2\log n+B^{1+\alpha_1}\big]^2
\le
C_{p_0,g_0}\, t(\log n)^{\gamma_0},
\]
for some sufficiently large constant $C_{p_0,g_0}>0$. We choose
\[
\delta=tC_{p_0,g_0,A}\frac{(\log n)^{\gamma_0}}{n}.
\]
This completes the proof.
\end{proof}

\subsection{Proof of Theorem~\ref{thm:approximateNPMLE-hetero}}

Recalling the notation from Section~\ref{sec:heter}, we establish the following lemma.

\begin{lemma}\label{lem:analytic-log-density2}
Fix $g\in\cG$ and assume that
\[
M_g\coloneqq \sup_{\theta\in\mathrm{supp}(g)}|\theta|<\infty.
\]
For any $B_1,B_2>0$ and $\tau_0>1$, define
\[
\rho_1\coloneqq \delta_1+\sqrt{1+\delta_1^2},
\qquad
\delta_1\coloneqq \frac{\pi}{8B_1(\tau_0+2)B_2M_g}\land\frac{1}{2},
\]
and
\[
\rho_2\coloneqq \delta_2+\sqrt{1+\delta_2^2},
\qquad
\delta_2\coloneqq \frac{\pi}{4(4B_1+M_g)B_2M_g}\land\frac{1}{2}\land\sqrt{\Big(\frac{\tau_0+1}{2}\Big)^2-1}.
\]
Then $\log f_g(z|\tau)$ admits an analytic extension to the translated Bernstein polyellipse
\[
\mathcal{E}(B_1,B_2,\tau_0)\coloneqq E_{\rho_1}^{(B_1)}\times \Big[E_{\rho_2}^{(B_2)}+B_2\tau_0\Big].
\]
Moreover,
\[
\sup_{(z,\tau)\in\mathcal{E}(B_1,B_2,\tau_0)}|\log f_g(z|\tau)|
\le
C_{B_2,\tau_0}(B_1^2+M_g^2),
\]
for some constant $C_{B_2,\tau_0}>0$ depending only on $B_2$ and $\tau_0$, provided that $B_1M_g\ge 1$.
\end{lemma}

\begin{proof}
For $(z,\tau)\in\mathbb{C}\times \big(\mathbb{C}\backslash(-\infty,0]\big)$, write
\[
f_g(z|\tau)\coloneqq \sqrt{\frac{\tau}{2\pi}}\exp\Big(-\frac{\tau z^2}{2}\Big)I_g(z,\tau),
\qquad
I_g(z,\tau)\coloneqq \int \exp\Big(\tau z\theta-\frac{\tau\theta^2}{2}\Big)\,\dif g(\theta).
\]
Moreover,
\[
I_g(z,\tau)
=
\int \exp\bigg(\Re\Big[\tau z\theta-\frac{\tau\theta^2}{2}\Big]\bigg)
\Big\{\cos\big(\alpha(\theta)\big)+i\sin\big(\alpha(\theta)\big)\Big\}\,\dif g(\theta),
\]
where
\[
\alpha(\theta)\coloneqq (\Re\tau)(\Im z)\theta+(\Im\tau)\Big((\Re z)\theta-\frac{\theta^2}{2}\Big).
\]

Now fix $(z,\tau)\in \mathcal{E}(B_1,B_2,\tau_0)$. Then
\[
\begin{aligned}
|\alpha(\theta)|
&\le B_2(\tau_0+\rho_2)B_1\delta_1M_g
   +B_2\delta_2\Big(B_1\rho_1M_g+\frac{M_g^2}{2}\Big)\\
&\le B_2(\tau_0+2)B_1\delta_1M_g
   +B_2\delta_2\Big(2B_1M_g+\frac{M_g^2}{2}\Big)\\
&\le \frac{\pi}{8}+\frac{\pi}{8}
= \frac{\pi}{4}.
\end{aligned}
\]
Hence $\Re[I_g(z,\tau)]>0$, and $I_g(z,\tau)$ stays in the right half-plane. Therefore we may take the principal branch of the logarithm and define
\[
\log f_g(z|\tau)
\coloneqq
\log I_g(z,\tau)-\frac{\tau z^2}{2}+\frac{\log(\tau)-\log(2\pi)}{2},
\]
which is jointly holomorphic on $\mathcal{E}(B_1,B_2,\tau_0)$.

It remains to bound its modulus. For $(z,\tau)\in\mathcal{E}(B_1,B_2,\tau_0)$,
\[
\bigg|\Re\Big[\tau z\theta-\frac{\tau\theta^2}{2}\Big]\bigg|
\le
|\tau|\Big(|z|M_g+\frac{M_g^2}{2}\Big)
\le
B_2(\tau_0+2)\Big(2B_1M_g+\frac{M_g^2}{2}\Big).
\]
Therefore,
\[
\begin{aligned}
|\log f_g(z|\tau)|
&\le
B_2(\tau_0+2)\Big(2B_1M_g+\frac{M_g^2}{2}\Big)
+\log\frac{1}{\cos(\pi/4)}
+\frac{\pi}{2}
+2B_2(\tau_0+2)B_1^2\\
&\quad
+\frac{|\log B_2|+\Big|\log\frac{\tau_0-1}{2}\Big|+\log(\tau_0+2)+\frac{\pi}{2}+\log(2\pi)}{2}\\
&\le
C_{B_2,\tau_0}(B_1^2+M_g^2),
\end{aligned}
\]
where we used $B_1M_g\ge 1$ in the last step. This completes the proof.
\end{proof}

\bigskip

\begin{proof}[Proof of Theorem~\ref{thm:approximateNPMLE-hetero}]
Note that
\[
\supp(P_{J_n})\subseteq \supp(P_n)\subseteq [-|X|_n,|X|_n]\times[1/T_0,T_0].
\]
By Proposition~25 of \citet{lindsay1995mixture}, both $\hat{g}$ and $\hat{g}_{J_n}$ are supported on
\[
[-(|X|_n\land M),\,|X|_n\land M].
\]

Proceeding as in the proof of Theorem~\ref{thm:approximateNPMLE1}, we obtain
\[
\sum_{i=1}^n\log f_{\hat{g}}(X_i|\tau_i)-\sum_{i=1}^n\log f_{\hat{g}_{J_n}}(X_i|\tau_i) 
\le
4n\sup_{g\in\{\hat{g},\hat{g}_{J_n}\}}
\|\log f_g-\mathrm{poly}_{J_n}[\log f_g]\|_{\infty,B},
\]
where $\|\cdot\|_{\infty,B}$ and $\mathrm{poly}_{J_n}[\cdot]$ are defined as in Theorem~\ref{thm:Bernstein-multi}, with the translated Bernstein polyellipse from Lemma~\ref{lem:analytic-log-density2}. For that polyellipse, we take $B_1=|X|_n$ and fix $B_2,\tau_0$ so that
\[
[B_2(\tau_0-1),\,B_2(\tau_0+1)]=[1/T_0,T_0].
\]
Applying \eqref{eq:Bernstein-multi} together with the bounds from Lemma~\ref{lem:analytic-log-density2}, we obtain
\[
\sup_{g\in\{\hat{g},\hat{g}_{J_n}\}}
\|\log f_g-\mathrm{poly}_{J_n}[\log f_g]\|_{\infty,B}
\le
\frac{64C_{B_2,\tau_0}|X|_n^2}{\delta_n^2(1+\delta_n)^{\lfloor J_n/2\rfloor+1}},
\]
where
\[
\delta_n=
\frac{\pi}{8(\tau_0+2)B_2|X|_n(|X|_n\land M)}
\land \frac{1}{2}
\land \sqrt{\Big(\frac{\tau_0+1}{2}\Big)^2-1}.
\]
Therefore, to ensure the desired bound, it suffices to take
\[
J_n=\Big\lceil
\frac{2}{\log(1+\delta_n)}
\log\Big(\frac{256n^{1+\gamma}C_{B_2,\tau_0}|X|_n^2}{\delta_n^2}\Big)\Big\rceil.
\]
The asserted order of $J_n$ now follows immediately.
\end{proof}

\subsection{Auxiliary results}

Our first auxiliary lemma clarifies how scenarios~\ref{(S1)} and~\ref{(S2)} imply the light-tail condition~\ref{(LT)}.

\begin{lemma}\label{lem:tail}
Under scenario~\ref{(S1)}, condition~\ref{(LT)} holds with $\beta_0=1$. Under scenario~\ref{(S2)}, if $g_0$ is sub-Weibull in the sense that there exists $\beta>0$ such that
\[
\int \mathbf{1}\{|\theta|\ge t\}\,\mathrm{d}g_0(\theta)\le \exp(-c_4 t^\beta),
\qquad t\ge c_3,
\]
for some $c_3,c_4>0$, then condition~\ref{(LT)} holds with $\beta_0=\min(\alpha_2\beta,1+\alpha_2)$.
\end{lemma}

\begin{proof}
Under scenario~\ref{(S1)}, for any $t>0$,
\[
\int \ind\{|x|\ge t\} f_{g_0}(x)\,\dif\mu(x)
\le
\exp(-\tau t)\int \exp(\tau|x|)f_{g_0}(x)\,\dif\mu(x).
\]
Moreover,
\[
\begin{aligned}
\int \exp(\tau|x|)p_{\theta}(x)\,\dif\mu(x)
&\le
\int \big(\exp(\tau x)+\exp(-\tau x)\big)p_{\theta}(x)\dif\mu(x)\\
&=
\exp\big(\kappa(\theta+\tau)-\kappa(\theta)\big)
+\exp\big(\kappa(\theta-\tau)-\kappa(\theta)\big),
\end{aligned}
\]
and hence
\[
\int \exp(\tau|x|)f_{g_0}(x)\,\dif\mu(x)
\le
2\exp\Big(2\sup_{|\theta|\le M+\tau}|\kappa(\theta)|\Big).
\]
Therefore,
\[
\int \ind\{|x|\ge t\} f_{g_0}(x)\,\dif\mu(x)
\le
2\exp\Big(2\sup_{|\theta|\le M+\tau}|\kappa(\theta)|-\tau t\Big).
\]
Consequently,
\[
\bP\big(|X|_n\ge t\big)
\le
2n\exp\Big(2\sup_{|\theta|\le M+\tau}|\kappa(\theta)|-\tau t\Big),
\]
which verifies \ref{(LT)} with $\beta_0=1$.

Under scenario~\ref{(S2)}, for any $t>0$ and $s\ge c_3$,
\[
\begin{aligned}
\int \ind\{|x|\ge t\} f_{g_0}(x)\,\dif\mu(x)
&\le
\int \ind\{|\theta|\ge s\}\,\dif g_0(\theta)
+\sup_{|\theta|\le s}\int \ind\{|x|\ge t\}p_\theta(x)\,\dif\mu(x)\\
&\le
\exp(-c_4s^\beta)
+\exp(-st)\sup_{|\theta|\le s}\int \exp(s|x|)p_\theta(x)\,\dif\mu(x)\\
&\le
\exp(-c_4s^\beta)
+2\exp\Big(2\sup_{|\theta|\le 2s}\kappa(\theta)-st\Big)\\
&\le
\exp(-c_4s^\beta)
+2\exp\Big(2c_2(2s\lor c_0)^{1+1/\alpha_2}-st\Big).
\end{aligned}
\]
Taking $s=ct^{\alpha_2}$ for a sufficiently small constant $c>0$, we obtain~\ref{(LT)} with
\[
\beta_0=\min(\alpha_2\beta,1+\alpha_2).
\]
This completes the proof.
\end{proof}

In scenario~\ref{(S1)}, \eqref{eq:bound} holds with $M_n=M$. An analogous result holds under scenario~\ref{(S2)}.

\begin{lemma}\label{lem:scenario2}
Under scenario~\ref{(S2)}, if $|X|_n\ge 1$\footnote{The condition $|X|_n\ge 1$ is imposed only for simplicity of presentation, and will be omitted in subsequent applications of the lemma.}, then \eqref{eq:bound} holds with
\[
M_n=C_{p_0}|X|_n^{\alpha_1},
\]
where $C_{p_0}>0$ is a constant depending only on $p_0$.
\end{lemma}

\begin{proof}
By assumption, $\kappa(\theta)$ is nonnegative and strictly convex on $\mathbb{R}$. Define
\[
\mu(\theta)\coloneqq \kappa'(\theta).
\]
For $\theta\ge c_0$,
\[
c_1|\theta|^{1/\alpha_1}
\le \frac{\kappa(\theta)-\kappa(0)}{\theta}
\le \mu(\theta)
.
\]
Similarly, for $\theta\le -c_0$,
\[
 \mu(\theta)
\le \frac{\kappa(0)-\kappa(\theta)}{-\theta}
\le -\,c_1|\theta|^{1/\alpha_1}.
\]
Hence there exist constants $b_0,b_1>0$, depending only on the constants above, such that
\[
|\mu^{-1}(x)|\le b_1|x|^{\alpha_1},\qquad |x|\ge b_0.
\]

Recalling the argument at the beginning of the proof of Theorem~\ref{thm:approximateNPMLE1}, we have
\[
\mathrm{supp}(P_{J_n})\subseteq \mathrm{conv}\big(\mathrm{supp}(P_n)\big)\subseteq[-|X|_n,|X|_n].
\]
For each $x\in\mathbb{R}$, the likelihood kernel $l_\theta(x)$ is unimodal in $\theta$, with unique mode at $\mu^{-1}(x)$. Therefore, by Proposition 25 of \citet{lindsay1995mixture}, the supports of $\hat{g}$ and $\hat{g}_{J_n}$ are both contained in
\[
\big[\mu^{-1}(-|X|_n),\mu^{-1}(|X|_n)\big].
\]
It follows that \eqref{eq:bound} holds with
\[
M_n=b_1(|X|_n\lor b_0)^{\alpha_1}\le C_{p_0}|X|_n^{\alpha_1},
\]
where $C_{p_0}$ depends only on $b_0$, $b_1$, and $\alpha_1$. This completes the proof.
\end{proof}

\end{document}